\documentclass[12pt]{amsart}
\usepackage[utf8]{inputenc}
\usepackage[T1]{fontenc}
\usepackage{lmodern}
\usepackage{color,graphicx,array, amssymb, amscd,slashed,bm}
\usepackage{comment}
\usepackage{hyperref}
\usepackage{graphicx}
\usepackage{subcaption}
\usepackage{mathrsfs}
\newtheorem{Theorem}{Theorem}[section]
\newtheorem{Lemma}[Theorem]{Lemma}
\newtheorem{Proposition}[Theorem]{Proposition}
\newtheorem{Corollary}[Theorem]{Corollary}
\theoremstyle{definition}
\newtheorem{Definition}[Theorem]{Definition}

\newtheorem{Remark}[Theorem]{Remark} 

\numberwithin{equation}{section}
\newcommand{\R}{\mathbb R}

\newcommand{\id}{\operatorname{Id}}

\usepackage{amsmath}	
\usepackage{cases}
\begin{document}

\title[]{$\mathbb{R}$-equivariant minimal Lagrangian surfaces in the complex quadric $Q_{2}$ and the complex hyperbolic quadric $Q_{2}^{*}$} 
 \author[S.~Zeng]{Sihao Zeng}
 \address{Department of Mathematics, Hokkaido University, 
 Sapporo, 060-0810, Japan}
 \email{sihao.zeng.f1@elms.hokudai.ac.jp}
 \thanks{The author is supported by Special Program: International Graduate Course for Data-driven and Hypothesis-driven Science (IGC-DHS)}

 \date{\today}
\pagestyle{plain}
\begin{abstract}
We study global and variational properties of
$\mathbb R$-equivariant minimal Lagrangian surfaces in the complex
quadric $Q_2$ and the complex hyperbolic quadric $Q_2^*$ arising from
loop group potentials. In the $Q_2^*$ case, we work under the
timelike-axis condition throughout the associated $\mathbb{S}^1$-family.
Excluding the flat case in $Q_2$, we give a necessary and
sufficient criterion for closing in the $y$-direction: 
a member $f^{\lambda_0}$ closes if and only if the
spectral ratio $\mu_2/\mu_1$ is a positive rational number.
The closing parameters form a dense countable subset of $\mathbb{S}^1$,
while the non-closing parameters are dense and have full arc-length
measure. The closing members yield cylindrical quotients in $Q_2$
and annular quotients in $Q_2^*$. Apart from the flat case in $Q_2$, the
surfaces under consideration have infinite absolute total curvature.
The equivariant surfaces considered here in $Q_2$ are unstable,
whereas minimal Lagrangian surfaces in $Q_2^*$ are stable. Finally, for compact cylindrical domains of the $y$-closing surfaces in $Q_2$, we obtain explicit sufficient width conditions for Hamiltonian
stability and instability under variations fixed near the boundary. We also establish the existence of a finite geometric critical width and determine it explicitly in the flat case.
\end{abstract}

\dedicatory{Dedicated to Professor Jun-ichi Inoguchi\\
on the occasion of his sixtieth birthday}

\maketitle

\section*{Introduction}

Minimal Lagrangian surfaces provide a natural link between
symplectic geometry, minimal surface theory, and integrable systems.
The complex quadric \(Q_2\) and the complex hyperbolic quadric
\(Q_2^*\) form a compact--noncompact pair \cite{CU,GVWX,WV2021}:
\[
    Q_2\cong \mathbb{S}^2\times \mathbb{S}^2,
    \qquad
    Q_2^*\cong \mathbb{H}^2\times \mathbb{H}^2.
\]
Although the surface theories in these two spaces admit parallel
loop group descriptions, their global and variational properties are
strongly influenced by the opposite signs of the ambient Ricci
curvature.

Related loop group constructions for minimal Lagrangian surfaces
in $\mathbb{C}P^2$ were developed by Dorfmeister and Ma in the
contractible case \cite{DM2021}.
They subsequently extended this framework to general Riemann
surfaces and constructed a class of minimal Lagrangian cylinders
by perturbing equivariant potentials \cite{DM2025}.

In earlier work \cite{KZ2026,KZ2026agag}, minimal Lagrangian surfaces in \(Q_2\) and \(Q_2^*\)
were constructed by the loop group method, and associated
\(\mathbb{S}^1\)-families of isometric deformations were obtained. The purpose of the present paper is to investigate the global and variational properties of the $\mathbb R$-equivariant surfaces arising
from these constructions. Our main results concern closing members of
the associated families, absolute total curvature, ordinary stability,
and explicit width conditions for Hamiltonian stability and instability
on compact cylindrical domains of the $y$-closing surfaces in $Q_2$.

In the $Q_2^*$ case, we restrict to the geometrically natural
regime in which the two linear equivariance actions associated
with the $\mathbb H^2$-factors have timelike axes throughout
the associated $\mathbb{S}^1$-family, see Remark~\ref{Rmk: timelike axes} for details. This preserves their causal type along the family and yields the real
elliptic-function description used in the global analysis below.

Our first main result concerns the closing problem. 
In earlier work \cite{KZ2026, KZ2026agag}, we provided sufficient
closing conditions for a prescribed period. Here we prove
necessity as well, determine the least positive closing period,
and describe the closing parameters throughout the associated
family. Write $\{f^\lambda\}_{\lambda\in \mathbb{S}^1}$ for the associated family, and let
$\mu_1(\lambda),\mu_2(\lambda)>0$ denote the two spectral quantities
defined in Section~\ref{Equivariant setup}. Under the nondegeneracy assumption in the $Q_2$
case and the standing timelike-axis assumption in the $Q_2^*$ case,
we prove that, for every $\lambda_0\in \mathbb{S}^1$,
\[
    f^{\lambda_0}\text{ closes in the $y$-direction}
    \quad\Longleftrightarrow\quad
    \frac{\mu_2(\lambda_0)}{\mu_1(\lambda_0)}
    \in\mathbb Q_{>0}.
\]
The closing parameters form a dense countable subset of $\mathbb{S}^1$,
and hence have zero arc-length measure, whereas the non-closing
parameters are dense and have full arc-length measure. The closing
members descend to cylindrical quotients in $Q_2$ and annular
quotients in $Q_2^*$.

We next study the global geometry of the induced metrics. In the
$Q_2$ case, the equivariant surfaces, as well as their closing
quotients, are complete. Moreover, apart from the flat case in $Q_2$ and under the standing assumption in the
$Q_2^*$ case, we prove that
\[
    \int_M |K|\,dA=+\infty.
\]
The mechanisms responsible for this divergence are different in the two spaces: in $Q_2$ it follows from
the nonconstant periodic behavior of the metric, whereas in $Q_2^*$
it is caused by the behavior of the conformal factor near the ends
of the maximal equivariant domain.

The two ambient spaces also exhibit opposite ordinary stability behavior. Here, stability is understood with respect to compactly supported normal variations:
\[
\begin{array}{ll}
Q_2:
&
\text{every $\mathbb R$-equivariant minimal Lagrangian surface
considered here is unstable},
\\[1mm]
Q_2^*:
&
\text{every minimal Lagrangian surface is stable}.
\end{array}
\]
The $Q_2^*$-statement follows from the general stability theorem for minimal
Lagrangian submanifolds in K\"ahler manifolds with non-positive Ricci curvature \cite{CLN2013,Le1989,Oh1990}, whereas the $Q_2$-statement follows
from completeness and the conformal-type obstruction of Miyaoka and
Ueki \cite{MU2014}. The former does not require the timelike-axis
assumption.

Finally, we study Hamiltonian stability and instability on the compact
cylindrical domains $\Omega_{\ell,L}$ of the $y$-closing
surfaces in $Q_2$ described in Section~\ref{sec: Hamiltonian stability}.
We consider Hamiltonian variations fixed in a neighbourhood of
the boundary. For a fixed left endpoint $x_1\in\mathbb R$ and a
closing period $L>0$, write
$\Omega_{\ell,L}=[0,\ell]\times(\mathbb R/L\mathbb Z)$ in the
translated coordinate $s=x-x_1$. Here $\ell>0$ is the width in
the specified conformal coordinate. Set $S:=a^2+b^2+c^2$,
where $a,b\in\mathbb R^*$ and $c\in\mathbb R$ are the parameters
of the equivariant potential. More precisely, Theorem~\ref{thm:direct-hamiltonian-widths} gives
\[
    0<\ell\leq\frac{\pi}{2\sqrt{2S}}
    \quad\Longrightarrow\quad
    \Omega_{\ell,L}\text{ is Hamiltonian stable},
\]
whereas
\[
    \ell>\frac{\pi}{4\sqrt{|ab|}}
    \quad\Longrightarrow\quad
    \Omega_{\ell,L}\text{ is Hamiltonian unstable}.
\]
The proof uses the Hamiltonian second variation and the sharp
Poincar\'e inequality on an interval, together with its Rayleigh
variational characterization. For each fixed immersion, left endpoint $x_1$ and closing period $L$,
Corollary~\ref{cor:geometric-critical-width} establishes a unique
finite geometric critical width $\ell_H=\ell_H(x_1,L)>0$ such that
\[
    \Omega_{\ell,L}\text{ is Hamiltonian stable}
    \quad\Longleftrightarrow\quad
    0<\ell\leq\ell_H,
\]
with
\[
    \frac{\pi}{2\sqrt{2S}}
    \leq\ell_H\leq
    \frac{\pi}{4\sqrt{|ab|}}.
\]
These bounds are uniform in $x_1$ and $L$, but do not determine
$\ell_H$ explicitly in general. In the flat case $c=0$ and
$|a|=|b|$, the bounds coincide and give
$\ell_H=\pi/(4|a|)$.

The paper is organized as follows. Section~\ref{Equivariant setup} recalls the equivariant
setup and the monodromy action used in the closing problem. Section~\ref{sc: R equivariant} studies the associated families and establishes the necessary and sufficient closing criterion together with the distribution of the closing parameters. Section~\ref{Absolute total curvature and stability} treats absolute total curvature and ordinary stability. Section~\ref{sec: Hamiltonian stability} establishes explicit width conditions for Hamiltonian stability and instability of compact cylindrical domains, proves the existence of a geometric critical width, and determines it explicitly in the flat case.

{\bf Acknowledgments:}
The author would like to express his sincere gratitude to his supervisor Prof. Shimpei Kobayashi for inspiring his interest in this problem and for many helpful discussions.

\section{Equivariant setup}\label{Equivariant setup}
We recall only the parts of the loop group construction that will be
used in the sequel. Define
 \[
\epsilon =
\begin{cases}
1, & \text{if the ambient space is } Q_{2},\\
-1, & \text{if the ambient space is } Q_{2}^{*}.
\end{cases}
\]
We use the product metric normalizations for which each $\mathbb{S}^2$-factor
of $Q_2\cong \mathbb{S}^2\times \mathbb{S}^2$ has Gaussian curvature $4$, while each
$\mathbb{H}^2$-factor of $Q_2^*\cong \mathbb{H}^2\times \mathbb{H}^2$ has Gaussian curvature
$-4$. Consequently,
\[
    \operatorname{Ric}_{Q_2}=4g_{Q_2},
    \qquad
    \operatorname{Ric}_{Q_2^*}=-4g_{Q_2^*}.
\]

As in the constructions of \cite{KZ2026,KZ2026agag}, the associated family is described by a pair of extended frames $(F_{\lambda}, F_{-\epsilon i\lambda})$. Set 
 \begin{equation}\label{eq:XY-associated-family}
     X^{\lambda} = (X^{\lambda}_{ij}) := F_{\lambda}F_{-\epsilon i\lambda}^{-1} , \quad Y^{\lambda} = (Y^{\lambda}_{ij}) := i F_{\lambda} \sigma_{3} F_{-\epsilon i\lambda}^{-1}, \quad \sigma_{3} = \begin{pmatrix} 1 & 0 \\
 0 & -1
\end{pmatrix}.
 \end{equation} 
Then the associated family $\{f^{\lambda} \}$ of a minimal Lagrangian surface is recovered as
\begin{equation}\label{eq:reconstruction-f-lambda}
    f^{\lambda}
    =
    \left[
    \begin{array}{c}
        \operatorname{Re}(X_{11}^{\lambda})
        + i\operatorname{Re}(Y_{11}^{\lambda})
        \\[1mm]
        \operatorname{Im}(X_{11}^{\lambda})
        + i\operatorname{Im}(Y_{11}^{\lambda})
        \\[1mm]
        \operatorname{Re}(X_{kl}^{\lambda})
        + i\operatorname{Re}(Y_{kl}^{\lambda})
        \\[1mm]
        \operatorname{Im}(X_{kl}^{\lambda})
        + i\operatorname{Im}(Y_{kl}^{\lambda})
    \end{array}
    \right],
\end{equation}
where $(k, l) = (1,2)$ in $Q_{2}$ and $(k, l) = (2,1)$ in $Q_{2}^{*}$. 
Here the brackets denote the corresponding projective class in $\mathbb{CP}^{3}$.

For later use, we recall the relation between the harmonic-map data
and the induced metric \cite{KZ2026,KZ2026agag}. Write
\[
    F_{\lambda}^{-1}dF_{\lambda}
    =
    \hat U_{1}^{\lambda}\,dz
    +
    \hat V_{1}^{\lambda}\,d\bar z .
\]
In a suitable diagonal unitary gauge, the $(1,0)$-part takes the form
\begin{equation}\label{eq:Uhat1}
    \hat U_{1}^{\lambda}
    =
    \begin{pmatrix}
        \dfrac14\hat u_{z}
        &
        -\dfrac{\epsilon}{\sqrt{2}}\lambda^{-1}
        e^{\hat u/2}
        \\[2mm]
        \dfrac{1}{\sqrt{2}}\lambda^{-1}
        \hat\alpha e^{-\hat u/2}
        &
        -\dfrac14\hat u_{z}
    \end{pmatrix},
\end{equation}
where $\hat u$ is real and
$\hat\alpha\,dz^{2}$ is the corresponding holomorphic quadratic
differential. The induced metric is related to these data by
\begin{equation}\label{eq:metric-reconstruction}
    ds_{M}^{2}
    =
    2e^{u}\,dz\,d\bar z
    =
    \left(
        e^{\hat u}
        +
        |\hat\alpha|^{2}e^{-\hat u}
    \right)dz\,d\bar z .
\end{equation}

We also recall the basic terminology for Lagrangian variations.
We assume that $N$ is K\"ahler, with complex structure $J$ and
K\"ahler form $\omega$, and that $f:M\to N$ is Lagrangian. Associated with
a normal vector field $V$ is the one-form
\[
\alpha_{V}:=f^{*}(V \rfloor \omega), \qquad \alpha_{V}(X)
=
\omega(V,df(X))
\]
for $X\in TM$. The normal vector field $V$ is called \textit{Lagrangian} if $\alpha_V$ is closed,
\[
    d\alpha_V=0,
\]
and \textit{Hamiltonian} if $\alpha_V$ is exact,
\[
    \alpha_V=df_V
\]
for some smooth function $f_V$. Throughout this paper, Hamiltonian stability is understood
with respect to Hamiltonian normal vector fields along the immersion. On a compact domain $\Omega \subset M$ with smooth boundary, we consider variations that are fixed in a neighborhood of $\partial \Omega$. Their normal variational fields satisfy
\[
\operatorname{supp} V \Subset \Omega^\circ.
\]
For a Hamiltonian normal field with $\alpha_{V} = df_{V}$, this is equivalent to
\[
\operatorname{supp} (df_{V}) \Subset \Omega^\circ.
\]
The potential $f_{V}$ itself need not have compact support; it is constant near each boundary component, and these constants need not agree.

Suppose in addition that $f$ is minimal and that the ambient space $(N,g)$ is K\"ahler–Einstein, with $\operatorname{Ric}_{N}= \kappa g$. The second variation formula of \cite[Theorem 3.5]{Oh1990} then gives
\begin{equation}\label{eq: second variation formula}
    \mathscr{A}''_{V}(0) = \int_{\Omega}\left\langle \left( \Delta_{H} - \kappa \right)df_{V}, df_{V}  \right\rangle dA = \int_{\Omega} \left( (\Delta_{H}f_{V})^{2} - \kappa |df_{V}|^{2} \right)dA. 
\end{equation}
Here $\Delta_{H} = d\delta + \delta d$ denotes the nonnegative Hodge Laplacian, acting on the appropriate degree. Indeed, 
\begin{equation*}
    \int_{\Omega}\left\langle \Delta_{H}df_{V}, df_{V} \right\rangle dA = \int_{\Omega} \left( |d(df_{V})|^{2} + |\delta df_{V}|^{2} \right)dA = \int_{\Omega}\left( \Delta_{H}f_{V} \right)^{2} dA.
\end{equation*}
No boundary term occurs because $df_{V}$ vanishes in a neighborhood of $\partial \Omega$. 

For $\theta>0$, we say that $f^\lambda$ \textit{closes in the
$y$-direction} with period $\theta$ if
\[
    f^\lambda(z+i\theta)=f^\lambda(z)
\]
for every $z$ in the equivariant domain $D$. Let
\[
    \mathfrak f^\lambda:
    D\longrightarrow
    \begin{cases}
        \mathbb S^7\subset\mathbb C^4,
        & \text{in the $Q_2$ case},\\[2pt]
        \mathbb H^7_3(-1) \subset\mathbb C^4_2,
        & \text{in the $Q_2^*$ case}
    \end{cases}
\]
denote the normalized lift obtained by multiplying the column
vector in the reconstruction formula above by $1/\sqrt{2}$, so that
$f^\lambda=[\mathfrak f^\lambda]$.
Since $f^\lambda$ is projective, the closing condition is
equivalent to the existence of a function
$\gamma:D\to\mathbb C^*$ such that
\[
    \mathfrak f^\lambda(z+i\theta)
    =
    \gamma(z)\mathfrak f^\lambda(z)
    \qquad\text{for all }z\in D.
\]
Throughout this paper, quotient surfaces refer to quotients arising from closing in the $y$-direction; possible further quotients by additional independent periods are not considered.

\begin{Definition}[$\mathbb R$-equivariant potentials, \cite{SKKR,BRS2010}]
 Define 
 \begin{equation}\label{eq: Delaunay potential}
    \xi = A(\lambda) \, d z, \quad \text{where} \quad A(\lambda) = \begin{pmatrix}
 c & a\lambda^{-1} + b\lambda \\
 \epsilon (a\lambda + b\lambda^{-1}) & -c
\end{pmatrix}, 
\end{equation}
with $a, b \in \R^{*}$ and $c \in \R$.  
We call such potentials \textit{equivariant potentials}.
\end{Definition}
For a fixed \(\lambda_0\in \mathbb{S}^1\), the squared eigenvalues of
\(A(\lambda_0)\) and \(A(-\epsilon i\lambda_0)\) are
\[
    \mu_1^2
    =
    c^2+\epsilon
    |\lambda_0a+\lambda_0^{-1}b|^2,
    \qquad
    \mu_2^2
    =
    c^2+\epsilon
    |\lambda_0a-\lambda_0^{-1}b|^2.
\]
Equivalently,
\begin{equation*}
\mu_{1}^{2} = c^{2} + \epsilon a^{2} + \epsilon b^{2} + \epsilon(\lambda_{0}^{2} + \lambda_{0}^{-2})a b, \quad \mu_{2}^{2} = c^{2} + \epsilon a^{2} + \epsilon b^{2} - \epsilon (\lambda_{0}^{2} + \lambda_{0}^{-2})a b.
\end{equation*}
Whenever these quantities are positive, we denote by
\(\mu_1,\mu_2>0\) their positive square roots.

For later use in the closing problem, we record the transformation
of the extended frame under translation in the equivariant
direction. For $\theta\in\mathbb R$, let
\[
    \gamma_{\theta}(z)
    :=
    z+i\theta .
\]
Since the solution of $d\Phi=\Phi\xi,~ \Phi(0,\lambda)=I$ is
\[
    \Phi(z,\lambda)
    =
    \exp\bigl(zA(\lambda)\bigr),
\]
the equivariance of the Iwasawa factor gives
\begin{equation}\label{eq: equivariance relation}
    F_{\lambda}(z+i\theta)
    =
    \exp\bigl(i\theta A(\lambda)\bigr)
    F_{\lambda}(z), \quad F_{-\epsilon i\lambda}(z+i\theta)
    =
    \exp\bigl(i\theta A(-\epsilon i\lambda)\bigr)
    F_{-\epsilon i\lambda}(z).
\end{equation}
Thus, if
\[
    M_{1}(\theta,\lambda)
    :=
    \exp\bigl(i\theta A(\lambda)\bigr),
    \qquad
    M_{2}(\theta,\lambda)
    :=
    \exp\bigl(i\theta A(-\epsilon i\lambda)\bigr),
\]
then \eqref{eq:XY-associated-family} gives
\begin{equation}\label{eq:translation-XY}
    X^{\lambda}(z+i\theta)
    =
    M_{1}X^{\lambda}(z)M_{2}^{-1},
    \qquad
    Y^{\lambda}(z+i\theta)
    =
    M_{1}Y^{\lambda}(z)M_{2}^{-1}.
\end{equation}
Consequently, the closing problem for $f^{\lambda}$ reduces to the
action of the two monodromy matrices
\[
    \exp\bigl(i\theta A(\lambda)\bigr),
    \qquad
    \exp\bigl(i\theta A(-\epsilon i\lambda)\bigr)
\]
on the reconstruction \eqref{eq:reconstruction-f-lambda}.

\section{The associated $\mathbb{S}^{1}$-family of $\mathbb{R}$-equivariant minimal Lagrangian surfaces in $Q_{2}$ and $Q_{2}^{*}$}\label{sc: R equivariant}
We now recall the metric data for the equivariant families and study
their closing behavior.

\subsection{$\mathbb R$-equivariant minimal Lagrangian surfaces in $Q_2$}
The following computation is based on the Delaunay-frame
construction in \cite[Theorem 3]{SKKR}. For completeness, we derive the Maurer--Cartan coefficients
and the profile equation directly from the twisted equivariant
potential \eqref{eq: Delaunay potential}, retaining the
parameter $c\in\mathbb R$. Write
\[
    F_{\lambda}^{-1}dF_{\lambda}
    =
    \Theta_{1}\,dx+\Theta_{2}\,dy .
\]
Differentiating \eqref{eq: equivariance relation} with respect to $\theta$ at $\theta=0$ gives
\[
    (F_{\lambda})_{y}
    =
    iA(\lambda)F_{\lambda}.
\]
Hence
\begin{equation}\label{eq:theta2-conjugacy}
    \Theta_{2}
    =
    F_{\lambda}^{-1}iA(\lambda)F_{\lambda},
\end{equation}
so that $\Theta_{2}$ and $iA(\lambda)$ have the same determinant. To obtain these components explicitly, use the normalized
Iwasawa decomposition
\[
    \Phi(z,\lambda)=F_\lambda(z)B_+(z,\lambda),
    \qquad
    B_+\in\Lambda^+\mathrm{SL}(2,\mathbb C)_\sigma,
\]
where $\Lambda^+\mathrm{SL}(2,\mathbb C)_\sigma$ denotes the
positive twisted loop subgroup. In particular, $B_+$ extends
holomorphically to $|\lambda|<1$ and satisfies
\[
    B_+(z,-\lambda)
    =
    \sigma_3B_+(z,\lambda)\sigma_3.
\]
By the normalization,
\[
    B_+(z,0)
    =
    \begin{pmatrix}
        r&0\\
        0&r^{-1}
    \end{pmatrix},
    \qquad r>0.
\]
Since
\[
    \Phi(z+i\theta,\lambda)
    =
    e^{i\theta A(\lambda)}\Phi(z,\lambda)
    =
    \bigl(e^{i\theta A(\lambda)}F_\lambda(z)\bigr)
    B_+(z,\lambda),
\]
and $e^{i\theta A(\lambda)}$ is an $SU(2)$-valued twisted
loop, the right-hand side is another normalized Iwasawa
decomposition of $\Phi(z+i\theta,\lambda)$.
Uniqueness therefore gives
\[
    B_+(z+i\theta,\lambda)=B_+(z,\lambda).
\]
Thus $B_+$ is independent of $y$, and in particular $r=r(x)$. Using $\Phi_z=\Phi A(\lambda)$ and $\Phi_{\bar z}=0$,
differentiation of $\Phi=F_\lambda B_+$ gives, in the
notation introduced in Section~\ref{Equivariant setup},
\[
    \hat U_{1}^\lambda
    =
    B_+A(\lambda)B_+^{-1}-(B_+)_zB_+^{-1},
    \qquad
    \hat V_{1}^\lambda
    =
    -(B_+)_{\bar z}B_+^{-1}.
\]
Since both $B_+$ and $B_+^{-1}$ extend holomorphically to
the unit disk, and $A(\lambda)$ has at most a simple pole at
$\lambda=0$, the form $\hat U_1^\lambda$ has no powers below
$\lambda^{-1}$, whereas $\hat V_1^\lambda$ has no negative
powers. On $\mathbb S^1$, the reality condition
\[
    \hat V_1^\lambda
    =
    -(\hat U_1^\lambda)^\dagger,
\]
where $\dagger$ denotes the conjugate transpose, together with $\bar\lambda=\lambda^{-1}$, excludes all
positive powers from $\hat U_1^\lambda$.
Thus $\hat U_1^\lambda$ contains only the powers
$\lambda^{-1}$ and $\lambda^0$. Its $\lambda^{-1}$ coefficient is obtained by conjugating
the corresponding coefficient of $A(\lambda)$ by $B_+(z,0)$.
Moreover, the constant coefficient of $\hat V_{1}^\lambda$
is $-(r'/2r)\sigma_3$. Hence
\begin{equation}\label{eq: hat U1 hat V1}
    \hat U_{1}^\lambda
    =
    \frac{r'}{2r}\sigma_3
    +
    \lambda^{-1}
    \begin{pmatrix}
        0&ar^2\\
        br^{-2}&0
    \end{pmatrix},
    \qquad
    \hat V_{1}^\lambda=-(\hat U_{1}^{\lambda})^{\dagger}.
\end{equation}
Define
\[
    v(x):=2b\,r(x)^{-2}, 
\]
and we refer to \(v\) as the equivariant profile function. Using \eqref{eq: hat U1 hat V1} and 
\[
    \Theta_1=\hat U_{1}^\lambda+\hat V_{1}^\lambda,
    \qquad
    \Theta_2=i(\hat U_{1}^\lambda-\hat V_{1}^\lambda),
\]
we obtain
\begin{equation}\label{eq:theta2-Q2}
\Theta_{1}=
    \begin{pmatrix}
        0
        &
        \dfrac{2ab}{\lambda v} - \dfrac{v\lambda}{2}
        \\[2mm]
        -\dfrac{2ab\lambda}{v} + \dfrac{v}{2\lambda}
        &
        0
    \end{pmatrix}, \quad
    \Theta_{2}
    =
    i
    \begin{pmatrix}
        -\dfrac{v'}{2v}
        &
        \dfrac{2ab}{\lambda v}
        +\dfrac{v\lambda}{2}
        \\[2mm]
        \dfrac{2ab\lambda}{v}
        +\dfrac{v}{2\lambda}
        &
        \dfrac{v'}{2v}
    \end{pmatrix}.
\end{equation}
 The normalization at $z=0$ gives $r(0)=1$ and $F_\lambda(0)=I$.
Thus $v(0)=2b$. Comparing the diagonal entries in
$\Theta_2(0)=iA(\lambda)$ further gives 
\[
    -\frac{v'(0)}{2v(0)}=c. 
\]
Hence $v'(0)=-4bc$. Taking determinants in \eqref{eq:theta2-conjugacy} and using
\eqref{eq:theta2-Q2}, we obtain
\begin{equation*}
    \det\Theta_{2}
    =
    \frac{(v')^{2}}{4v^{2}}
    +
    \frac{4a^{2}b^{2}}{v^{2}}
    +
    ab(\lambda^{2}+\lambda^{-2})
    +
    \frac{v^{2}}{4}.
\end{equation*}
On the other hand,
\[
    \det\bigl(iA(\lambda)\bigr)
    =
    c^{2}+a^{2}+b^{2}
    +
    ab(\lambda^{2}+\lambda^{-2}).
\]
Therefore,
\begin{equation}\label{vx1}
    (v')^{2}
    =
    -\bigl(v^{2}-4a^{2}\bigr)
     \bigl(v^{2}-4b^{2}\bigr)
    +
    4c^{2}v^{2}.
\end{equation}
The Maurer--Cartan equation 
\[
    (\Theta_2)_x+[\Theta_1,\Theta_2]=0
\]
also yields, from its diagonal component,
\[
    vv''-(v')^2=16a^2b^2-v^4.
\]
Combining this identity with \eqref{vx1}, we obtain
\begin{equation}
    v''
    =
    4(a^2+b^2+c^2)v-2v^3.
\end{equation}
Together with
\[
    v(0)=2b,\qquad v'(0)=-4bc,
\]
this second-order equation uniquely determines the frame
profile, and \eqref{vx1} is its first integral.
The resulting solution can be written in the form
\begin{equation}\label{eq: solution v}
    v(x) = \operatorname{sgn}(b) \sqrt{W_{2}} \operatorname{dn}\left(\sqrt{W_{2}}(x - x_{0}), \sqrt{1 - W_{1}/W_{2}} \right),
\end{equation}
where $\operatorname{dn}$ denotes the Jacobi elliptic function,
$\sqrt{1-W_1/W_2}$ is its modulus,
$x_0$ is chosen so that $v(0)=2b$ and $v'(0)=-4bc$, and
\begin{align*}
    W_{1} &= 2(a^{2} + b^{2} + c^{2}) - 2\sqrt{(a^{2} + b^{2} + c^{2})^{2} - 4a^{2}b^{2}}, \\
    W_{2} &= 2(a^{2} + b^{2} + c^{2}) + 2\sqrt{(a^{2} + b^{2} + c^{2})^{2} - 4a^{2}b^{2}}. 
\end{align*}

From \eqref{eq:theta2-Q2}, we obtain
\begin{equation}\label{eq:equivariant-Uhat}
    \frac12(\Theta_{1}-i\Theta_{2})
    =
    \begin{pmatrix}
        -\dfrac{v'}{4v}
        &
        \dfrac{2ab}{\lambda v}
        \\[2mm]
        \dfrac{v}{2\lambda}
        &
        \dfrac{v'}{4v}
    \end{pmatrix}.
\end{equation}
Since $dx=(1/2)(dz+d\bar z),~ dy=(-i/2)(dz-d\bar z)$, the left-hand side of \eqref{eq:equivariant-Uhat} is the $(1,0)$-part of the Maurer--Cartan form of the normalized equivariant frame. Comparing the absolute values of the upper-right entries of \eqref{eq:equivariant-Uhat} and
\eqref{eq:Uhat1}, and using $|\lambda|=1$, we obtain
\begin{equation}\label{eq:uhat-v}
    e^{\hat u}
    =
    8|ab|^{2}v^{-2}.
\end{equation}
Comparing the absolute values of the lower-left entries and using \eqref{eq:uhat-v}, we obtain
\begin{equation}\label{eq:alphahat-ab}
    |\hat\alpha|^{2}
    =
    4|ab|^{2}.
\end{equation}
Substituting \eqref{eq:uhat-v} and
\eqref{eq:alphahat-ab} into the metric relation
\eqref{eq:metric-reconstruction}, we obtain
\begin{equation}\label{eq: cylinder metric}
  ds^{2}_{M}
  =
  2e^{u}\, dz\,d\bar z
  =
  \frac12
  \left(v(x)^{2}+16|ab|^{2}\,v(x)^{-2}\right)dz\,d\bar z,
\end{equation}
where $v = v(x)$ is the solution \eqref{eq: solution v} of \eqref{vx1}. Furthermore, we can prove the completeness of this induced metric.
\begin{Proposition}\label{prop: complete}
The $\mathbb{R}$-equivariant minimal Lagrangian surfaces in $Q_{2}$ constructed from the $\mathbb{R}$-equivariant potentials \eqref{eq: Delaunay potential} are complete with respect to the induced metric \eqref{eq: cylinder metric}. Moreover, any quotient to which the immersion descends is also complete.
\end{Proposition}
\begin{proof}
By the AM--GM inequality,
\[
    \frac12
    \left(v^2+16|ab|^2v^{-2}\right)
    \geq 4|ab|>0.
\]
Hence, on the universal cover \(\widetilde M\simeq\mathbb C\),
\[
    ds_M^2\geq 4|ab|\,|dz|^2.
\]
Thus the induced metric is bounded below by a positive constant
multiple of the Euclidean metric and is therefore complete.

If the immersion descends to a quotient
\(M=\mathbb C/\Gamma\), then the deck transformations preserve the
induced metric. Hence the quotient map is a Riemannian covering, and
completeness of the covering metric implies completeness of the
quotient.
\end{proof}

\begin{Remark}\label{Rmk: flat case}
When $c=0$ and $|a|=|b|$, the profile $v\equiv2b$ is constant, and \eqref{eq: cylinder metric} gives
\[
    ds_M^2=4a^2|dz|^2.
\]
In particular, the Gaussian curvature $K\equiv0$. We refer to this as the flat case. By assuming $c^2+(|a|-|b|)^2>0$, we exclude this case below whenever
nondegeneracy is required.
\end{Remark}

\subsection{$\mathbb R$-equivariant minimal Lagrangian surfaces in $Q_{2}^{*}$} 
The construction of \cite[Theorem 5.1]{BRS2010} gives
\begin{equation}\label{vx2}
    (v')^2
    =(v^2-4a^2)(v^2-4b^2)+4c^2v^2. 
\end{equation}
Here $v$ is the profile of the normalized equivariant frame, with
\[
    v(0)=2b,\qquad v'(0)=-4bc.
\]
Let $ I=(-\kappa_1^2,\kappa_2^2)$ be the maximal interval containing $0$ on which \(v\) is finite and nonzero, and set
\[
    \Sigma :=I\times\mathbb R
    =
    \{x+iy\in\mathbb C:x\in I\}.
\]
The same computation as above, now with $\epsilon=-1$, gives
the same induced metric as \eqref{eq: cylinder metric}
\[
    ds^{2}_{M}
  =
  2e^{u}\, dz\,d\bar z
  =
  \frac12
  \left(v(x)^{2}+16|ab|^{2}\,v(x)^{-2}\right)dz\,d\bar z,
\]
where $v=v(x)$ denotes the real solution of \eqref{vx2} on the chosen maximal interval $I$. 

\begin{Remark}\label{Rmk: timelike axes}
    Unlike the case of $Q_2$, throughout the $Q_2^*$ case we
restrict ourselves to the equivariant potentials for which
the equivariance actions of the two associated spacelike
CMC surfaces have linear parts with timelike axes for every
$\lambda\in \mathbb{S}^1$. Equivalently, we assume that
    \[
        |\lambda a+\lambda^{-1}b|^{2}-c^{2}<0,
        \qquad
        |\lambda a-\lambda^{-1}b|^{2}-c^{2}<0
        \qquad\text{for all } \lambda\in \mathbb{S}^{1}.
    \]
    With the Lorentzian inner product $\langle X,Y\rangle=(1/2)\operatorname{tr}(XY)$ on
$\mathfrak{su}(1,1)\cong\mathbb R^{2,1}$, these inequalities
say that $iA(\lambda)$ and $iA(i\lambda)$, which represent
the axis directions of the two linear parts, are timelike, see \cite[Section~3.1 and Corollary~5.3]{BRS2010}.

    Since $a,b\in\mathbb R^{*}$, we have
    \[
        \max_{\lambda\in \mathbb{S}^{1}}
        |\lambda a\pm\lambda^{-1}b|^{2}
        =(|a|+|b|)^{2}.
    \]
    Hence the above assumption is equivalent to
    \[
        c^{2}>(|a|+|b|)^{2}.
    \]
    In particular,
    \[
        c^{2}-a^{2}-b^{2}>2|ab|.
    \]
    Under this assumption, the two projections onto the
$\mathbb H^2$-factors are the Gauss maps of spacelike
equivariant CMC surfaces whose equivariance actions have
linear parts with timelike axes throughout the associated
$\mathbb{S}^1$-family.
 
    This assumption also has a spectral role in the closing problem.
Indeed, for every $\lambda\in \mathbb{S}^1$,
\[
    \mu_j(\lambda)^2
    \ge
    c^2-(|a|+|b|)^2
    >0,
    \qquad j=1,2.
\]
Thus $A(\lambda)$ and $A(i\lambda)$ have distinct real eigenvalues throughout the associated family. Consequently, their monodromy matrices are diagonalizable over $\mathbb C$ with eigenvalues
\[
    e^{\pm i\theta\mu_j(\lambda)}.
\]
This property will be used in the proof of the closing criterion in
Theorem~\ref{thm:closing-criterion}.
\end{Remark}

Under the standing assumption, the real elliptic representation of
the profile function can be obtained directly from its differential
equation.

\begin{Lemma}[Real Jacobi representation]
\label{lem:real-sc}
Let $a,b\in\mathbb R^*$ and $c\in\mathbb R$ satisfy
\[
    c^2>(|a|+|b|)^2.
\]
Set
\[
    D:=c^2-a^2-b^2,
    \qquad
    Y_1:=2D-2\sqrt{D^2-4a^2b^2},
    \qquad
    Y_2:=2D+2\sqrt{D^2-4a^2b^2},
\]
and define
\[
    k:=\sqrt{1-\frac{Y_1}{Y_2}},
    \qquad
    \sigma:=-\operatorname{sgn}(bc).
\]
Then $0<Y_1<Y_2$ and $0<k<1$. The real solution of
\[
    (v')^2
    =(v^2-4a^2)(v^2-4b^2)+4c^2v^2,
    \qquad
    v(0)=2b,
    \qquad
    v'(0)=-4bc,
\]
is given by
\begin{equation}\label{eq: solution v2}
    v(x)
    =
    \sqrt{Y_1}\,
    \operatorname{sc}
    \left(
        \sigma\sqrt{Y_2}(x-x_0),k
    \right),
\end{equation}
where
$\operatorname{sc}(u,k):=\operatorname{sn}(u,k)/
\operatorname{cn}(u,k)$,
with $\operatorname{sn}$ and $\operatorname{cn}$ denoting
the Jacobi elliptic sine and cosine functions with modulus $k$,
respectively. The shift $x_0$ is chosen so that $v(0)=2b$.

The maximal interval $I$ containing $0$ on which $v$ is finite
and nonzero has two finite endpoints, one corresponding to
a zero and the other to an adjacent pole of the $\operatorname{sc}$-function.
\end{Lemma}

\begin{proof}
Since $D>2|ab|$, we have
\[
    0<Y_1<Y_2,
    \qquad
    Y_1+Y_2=4D,
    \qquad
    Y_1Y_2=16a^2b^2.
\]
Hence the differential equation becomes
\[
    (v')^2=(v^2+Y_1)(v^2+Y_2).
\]
For $s(u)=\operatorname{sc}(u,k)$, the Jacobi identities give
\[
    s'(u)
    =
    \frac{\operatorname{dn}(u,k)}
         {\operatorname{cn}^2(u,k)},
    \qquad
    (s')^2
    =
    (1+s^2)\bigl(1+(1-k^2)s^2\bigr).
\]
Let $w$ denote the right-hand side of \eqref{eq: solution v2}.
Using $1-k^2=Y_1/Y_2$, we obtain
\[
    (w')^2
    =
    Y_1Y_2
    \left(1+\frac{w^2}{Y_1}\right)
    \left(1+\frac{w^2}{Y_2}\right)
    =
    (w^2+Y_1)(w^2+Y_2).
\]

Write $K=K(k)$ for the complete elliptic integral of the first kind.
The function $\operatorname{sc}(\,\cdot\,,k)$ is a strictly
increasing bijection from $(-K,K)$ onto $\mathbb R$.
Choose $u_0\in(-K,K)$ such that
\[
    \operatorname{sc}(u_0,k)=\frac{2b}{\sqrt{Y_1}},
    \qquad
    x_0=-\frac{u_0}{\sigma\sqrt{Y_2}}.
\]
Then $w(0)=2b$. Moreover, $\operatorname{sgn}(w')=\sigma$, and
\[
    w'(0)
    =
    \sigma\sqrt{(4b^2+Y_1)(4b^2+Y_2)}
    =
    \sigma\,4|bc|
    =
    -4bc.
\]
Since $(v^2+Y_1)(v^2+Y_2)>0$, the prescribed initial derivative
fixes the sign of $v'$ throughout its interval of existence.
Thus uniqueness for the smooth first-order equation
\[
    v'
    =
    \sigma\sqrt{(v^2+Y_1)(v^2+Y_2)}
\]
shows that $v=w$.

Finally, on the real axis the zeros and poles of
$\operatorname{sc}(u,k)$ occur at $u=2mK$ and
$u=(2m+1)K$, respectively, with $m\in\mathbb Z$.
Restricting \eqref{eq: solution v2} to the component containing $x=0$
on which it is finite and nonzero therefore gives the stated
description of the endpoints of $I$.
\end{proof}

\subsection{Closing criterion and distribution of closing parameters}
We now give a necessary and sufficient condition for a member of the associated $\mathbb{S}^{1}$-family to close in the $y$-direction. The proof uses the projective reconstruction in Section \ref{Equivariant setup} and the action of the two monodromy matrices on $M_{2}(\mathbb{C})$. We also determine the least positive closing period and the distribution of closing parameters in the associated family.

\begin{Theorem}\label{thm:closing-criterion}
Let $\{f^\lambda\}_{\lambda\in \mathbb{S}^1}$ be the associated family of
$\mathbb R$-equivariant minimal Lagrangian immersions into $Q_2$ or $Q_2^*$
constructed from the equivariant potential \eqref{eq: Delaunay potential},
with $a,b\in\mathbb R^*$ and $c\in\mathbb R$. Assume, as in Remarks~\ref{Rmk: flat case}
and~\ref{Rmk: timelike axes}, that
\[
\begin{cases}
c^2+(|a|-|b|)^2>0, & \text{in the $Q_2$ case},\\
c^2>(|a|+|b|)^2, & \text{in the $Q_2^*$ case}.
\end{cases}
\]
Let $\mu_1(\lambda),\mu_2(\lambda)>0$ be given by
\[
\mu_1(\lambda)^2
=c^2+\epsilon|\lambda a+\lambda^{-1}b|^2,
\qquad
\mu_2(\lambda)^2
=c^2+\epsilon|\lambda a-\lambda^{-1}b|^2, 
\]
and let $D$ be the maximal equivariant domain. For any fixed
$\lambda_0\in \mathbb{S}^1$ and $\theta>0$, writing $\mu_j=\mu_j(\lambda_0)$,
we have
\begin{equation}
\label{eq:closing-period-iff}
f^{\lambda_0}(z+i\theta)=f^{\lambda_0}(z)
\quad\text{for all }z\in D
\quad\Longleftrightarrow\quad
\theta\mu_1,\ \theta\mu_2\in\pi\mathbb Z.
\end{equation}
Consequently, $f^{\lambda_0}$ closes in the $y$-direction if and only if
\begin{equation}
\label{eq:closing-rational-iff}
\frac{\mu_2}{\mu_1}\in\mathbb Q_{>0}.
\end{equation}
If $\mu_2/\mu_1=p/q$ with $p,q\in\mathbb Z_{>0}$ and $\gcd(p,q)=1$, then
its least positive $y$-closing period is
\begin{equation}
\label{eq:least-closing-period}
\theta_{\min}=\frac{q\pi}{\mu_1}=\frac{p\pi}{\mu_2},
\end{equation}
and every positive $y$-closing period is an integer multiple of
$\theta_{\min}$.

The set of closing parameters is dense and countable in $\mathbb{S}^1$, and hence
has zero standard arc-length measure. Its complement is dense and has
full arc-length measure. Each closing member descends to a quotient of
the maximal equivariant domain, namely a cylinder in the $Q_2$ case and
an annulus in the $Q_2^*$ case.
\end{Theorem}

\begin{proof}
Set $C=c^2+\epsilon(a^2+b^2)$. The assumptions give
$C>2|ab|$ in both cases. Thus, for every $\lambda\in \mathbb{S}^1$,
\[
\mu_j(\lambda)^2\ge C-2|ab|>0,\qquad j=1,2.
\]
Fix $\lambda_0\in S^1$ and $\theta>0$.
The maximal domain $D$ is $\mathbb C$ in the compact case and
$\Sigma$ in the hyperbolic case; in both cases it is connected
and invariant under translations in the $y$-direction. Put
\[
A_1=A(\lambda_0),\qquad A_2=A(-\epsilon i\lambda_0),\qquad
M_j= \exp\bigl(i\theta A_j\bigr) \quad(j=1,2). 
\]

\medskip
\noindent\textit{(a) Projective closing and a fixed eigenspace.}
Write $F_1=F_{\lambda_0}$, $F_2=F_{-\epsilon i\lambda_0}$, and set
\[
X=F_1F_2^{-1},\qquad
Y=F_1(i\sigma_3)F_2^{-1},\qquad
G=X+iY.
\]
Define the complex-linear isomorphisms
$\Psi_\pm:\mathbb C^4\longrightarrow M_2(\mathbb C)$ by
\[
\begin{aligned}
\Psi_+(w_1,w_2,w_3,w_4)
&=
\begin{pmatrix}
w_1+iw_2 & w_3+iw_4\\
-w_3+iw_4 & w_1-iw_2
\end{pmatrix},\\[1ex]
\Psi_-(w_1,w_2,w_3,w_4)
&=
\begin{pmatrix}
w_1+iw_2 & w_3-iw_4\\
w_3+iw_4 & w_1-iw_2
\end{pmatrix}.
\end{aligned}
\]
Using the respective matrix forms of $SU(2)$ and $SU(1,1)$, the
reconstruction formula in Section~\ref{Equivariant setup} gives
\[
w:=\Psi_\epsilon^{-1}(G),\qquad f^{\lambda_0}=[w],
\]
where $\Psi_\epsilon=\Psi_+$ for $\epsilon=1$ and
$\Psi_\epsilon=\Psi_-$ for $\epsilon=-1$. Moreover, we have
\[
G=F_1(I_2-\sigma_3)F_2^{-1}\ne0.
\]
The translation law from Section~\ref{Equivariant setup} therefore becomes
\begin{equation}
\label{eq:matrix-lift-translation}
G(z+i\theta)=T(G(z)),\qquad
T(Z):=M_1ZM_2^{-1}.
\end{equation}
Here $T$ is a fixed invertible complex-linear operator on
$M_2(\mathbb C)$.

Suppose that $\theta$ is a closing period. Then
$[T(G(z))]=[G(z)]$ in $\mathbb P(M_2(\mathbb C))$, so there is a unique
$\gamma(z)\in\mathbb C^*$ such that
\[
T(G(z))=\gamma(z)G(z).
\]
Locally, choose an entry $G_{rs}$ that does not vanish. Then
\[
    \gamma(z)=\frac{(T(G(z)))_{rs}}{G_{rs}(z)},
\]
so $\gamma$ is continuous. Since
$\gamma(D)\subset\operatorname{Spec}(T)$, the connectedness of
$D$ and finiteness of $\operatorname{Spec}(T)$ imply that
$\gamma\equiv\gamma_0$ for some
$\gamma_0\in\operatorname{Spec}(T)$. Consequently,
\[
G(D)\subset E,\qquad
E:=\ker(T-\gamma_0\operatorname{Id}). 
\]

\medskip
\noindent\textit{(b) A lower bound for the dimension of $E$.}
Differentiating $T(G)=\gamma_0G$ gives
\[
G,\quad G_z,\quad G_{\bar z}\in E.
\]
We claim that these three vectors are complex linearly independent.  We use the notation of \cite{KZ2026,KZ2026agag}. Let $\langle\,\cdot,\cdot\,\rangle$ denote the complex bilinear form on
$\mathbb C^4$ in the $Q_2$ case and on $\mathbb C^4_2$ in the $Q_2^*$
case, and let
\[
    (u,v):=\langle u,\overline v\rangle
\]
be the corresponding Hermitian form. 
Choose a local horizontal lift $\mathfrak f$ of $f^{\lambda_0}$
with $(\mathfrak f,\mathfrak f)=\epsilon$. Since $f^{\lambda_0}$ is conformal and Lagrangian with
induced metric $ds_M^2=2e^u\,dz\,d\bar z$, we have
\[
    (\mathfrak f_z,\mathfrak f_z)
    =
    (\mathfrak f_{\bar z},\mathfrak f_{\bar z})
    =
    e^u,
    \qquad
    (\mathfrak f_z,\mathfrak f_{\bar z})=0.
\]
Moreover, horizontality gives
\[
    (\mathfrak f_z,\mathfrak f)
    =
    (\mathfrak f_{\bar z},\mathfrak f)
    =0.
\]
The Gram matrix of $\mathfrak f,\mathfrak f_z,\mathfrak f_{\bar z}$ is
$\operatorname{diag}(\epsilon,e^u,e^u)$, which is nonsingular.
Hence these vectors are complex linearly independent. 

Since $w=\Psi_\epsilon^{-1}(G)$ and $\mathfrak f$ represent the
same projective immersion, write $w=r\mathfrak f$ locally,
where $r$ is smooth and nowhere zero.
Applying $\Psi_\epsilon^{-1}$ to a relation
$a_1G+a_2G_z+a_3G_{\bar z}=0$ and differentiating
$w=r\mathfrak f$ gives
\[
    (a_1r+a_2r_z+a_3r_{\bar z})\mathfrak f
    +a_2r\,\mathfrak f_z
    +a_3r\,\mathfrak f_{\bar z}=0.
\]
The preceding linear independence and $r\neq0$ imply
$a_1=a_2=a_3=0$. Thus $G,G_z,G_{\bar z}$ are linearly
independent, and their membership in $E$ gives
\begin{equation}
\label{eq:closing-eigenspace-bound}
    \dim_{\mathbb C}E\ge3.
\end{equation}

\medskip
\noindent\textit{(c) The spectrum of the monodromy action.}
Directly from the potential,
\[
\operatorname{tr}A_j=0,\qquad A_j^2=\mu_j^2\id.
\]
Since $\mu_j>0$, each $A_j$ has distinct eigenvalues
$\mu_j,-\mu_j$ and is diagonalizable over $\mathbb C$.
Set $\alpha=\theta\mu_1$, $\beta=\theta\mu_2$. There exist
$P_1,P_2\in GL(2,\mathbb C)$ such that
\[
M_j=P_jD_jP_j^{-1},\qquad
D_1=\operatorname{diag}(e^{i\alpha},e^{-i\alpha}),\quad
D_2=\operatorname{diag}(e^{i\beta},e^{-i\beta}).
\]
Define the complex-linear isomorphism
\[
    \mathcal P:M_2(\mathbb C)\longrightarrow M_2(\mathbb C),
    \qquad
    \mathcal P(W)=P_1WP_2^{-1}.
\]
Then
\[
    \widetilde T:=\mathcal P^{-1}\circ T\circ\mathcal P,
    \qquad
    \widetilde T(W)=D_1WD_2^{-1}.
\]
Thus $T$ and $\widetilde T$ have the same eigenvalues and
corresponding eigenspace dimensions.

For $W=(w_{rs})$, we have
\[
    \widetilde T(W)
    =
    \begin{pmatrix}
        e^{i(\alpha-\beta)}w_{11}
        &
        e^{i(\alpha+\beta)}w_{12}
        \\[1mm]
        e^{-i(\alpha+\beta)}w_{21}
        &
        e^{-i(\alpha-\beta)}w_{22}
    \end{pmatrix}.
\]
Thus the matrix units $E_{11},E_{12},E_{21},E_{22}$ form an eigenbasis
of $\widetilde T$. The eigenvalues of $T$, counted with multiplicity and arranged
by the corresponding matrix entries, are
\[
    \begin{pmatrix}
        e^{i(\alpha-\beta)}
        &
        e^{i(\alpha+\beta)}
        \\[1mm]
        e^{-i(\alpha+\beta)}
        &
        e^{-i(\alpha-\beta)}
    \end{pmatrix}.
\]
By \eqref{eq:closing-eigenspace-bound}, at least three entries of this
array must coincide. Any three positions in a $2\times2$ array contain
a complete row and a complete column. Equality along such a row gives
$e^{2i\beta}=1$, while equality along such a column gives
$e^{2i\alpha}=1$. Therefore
\[
\alpha,\beta\in\pi\mathbb Z.
\]
This proves the necessity in \eqref{eq:closing-period-iff}.

\medskip
\noindent\textit{(d) The converse and the least positive period.}
Conversely, suppose that $\theta\mu_1=m\pi$ and $\theta\mu_2=n\pi$.
Then
\[
M_1=(-1)^m\id,\qquad M_2=(-1)^n\id,
\]
and hence
\[
G(z+i\theta)=(-1)^{m-n}G(z).
\]
The scalar $(-1)^{m-n}$ does not change the projective class,
so $\theta$ is a closing period, proving \eqref{eq:closing-period-iff}. The signs of $M_1$ and $M_2$
need not agree.

Since $\theta,\mu_1,\mu_2>0$, the corresponding integers are positive,
and taking their ratio proves the necessity in
\eqref{eq:closing-rational-iff}. Conversely, if $\mu_2/\mu_1=p/q$,
then $\theta=q\pi/\mu_1$ satisfies the two integer conditions.

If $p,q$ are coprime, every closing period satisfies $n/m=p/q$, so
$m=kq$ and $n=kp$ for some $k\in\mathbb Z_{>0}$. Therefore
$\theta=k\theta_{\min}$, proving \eqref{eq:least-closing-period} and the
claim about all positive periods. For the specific lift $G$ used above,
\begin{equation}
\label{eq:lift-period-sign}
G(z+ik\theta_{\min})=(-1)^{k(q-p)}G(z).
\end{equation}
Thus $G$ is periodic at $\theta_{\min}$ when $p,q$ are both odd, and is
antiperiodic at $\theta_{\min}$ and periodic at $2\theta_{\min}$ when
they have opposite parity. 

\medskip
\noindent\textit{Distribution of the closing parameters.}
Write $\lambda=e^{i\varphi}$ and $t=\cos2\varphi$. Then
\[
\mu_1^2=C+2\epsilon abt,\qquad
\mu_2^2=C-2\epsilon abt.
\]
Consequently,
\[
l(t):=\sqrt{\frac{C-2\epsilon abt}{C+2\epsilon abt}},
\qquad
\frac{d}{dt}l(t)^2
=-\frac{4\epsilon abC}{(C+2\epsilon abt)^2}\ne0.
\]
Thus $l$ is continuous and strictly monotone on $[-1,1]$, with
nondegenerate image
\[
\left[
\sqrt{\frac{C-2|ab|}{C+2|ab|}},\,
\sqrt{\frac{C+2|ab|}{C-2|ab|}}
\right].
\]
By \eqref{eq:closing-rational-iff}, the closing parameter set is exactly
\[
\mathcal C
=\{e^{i\varphi}\in \mathbb{S}^1:l(\cos2\varphi)\in\mathbb Q_{>0}\}.
\]
For each positive rational number $r$, strict monotonicity gives at most
one $t$ with $l(t)=r$, and $\cos2\varphi=t$ has at most four solutions
on $\mathbb{S}^1$. Hence $\mathcal C$ is countable.

Let $U\subset \mathbb{S}^1$ be a nonempty open arc. The function
$\cos2\varphi$ is nonconstant on $U$, and $l$ is injective.
Therefore the continuous function $l(\cos2\varphi)$ is nonconstant on
$U$. Since $U$ is connected, its image is a nondegenerate interval,
which contains both rational and irrational numbers. Thus $U$ contains
both closing and non-closing parameters. This proves that $\mathcal C$
and $\mathbb{S}^1\setminus\mathcal C$ are dense.
Finally, a countable subset of $\mathbb{S}^1$ has zero standard arc-length measure,
so its complement has full arc-length measure.

The closing immersions descend to $D/(i\theta_{\min}\mathbb Z)$,
which is conformally a cylinder when $D=\mathbb C$ and an annulus
when $D=\Sigma$.
\end{proof}

\begin{Remark}
    In particular, $\alpha=\beta=\pi/2$ gives the eigenvalues $1,-1,-1,1$;
both eigenspaces have dimension two, so this case cannot be a closing
period.
\end{Remark}

\section{Absolute total curvature and stability}\label{Absolute total curvature and stability}
 In this section, we prove that, apart from the flat case in \(Q_2\) and under the standing assumption in the \(Q_2^*\) case, the equivariant surfaces have infinite absolute
total curvature. We then investigate their stability properties. It turns out that these surfaces in $Q_{2}$ are unstable, whereas minimal Lagrangian surfaces in $Q_{2}^{*}$ are stable. 

\subsection{Absolute total curvature of the $\R$-equivariant minimal Lagrangian surfaces}

First, recall that in the conformal coordinate $z = x + iy$, the induced metric is given by $ds^{2}_{M} = 2e^{u}dz d\bar{z}$. Then the Laplace operator $\Delta$ on the surface $(M , ds^{2}_{M})$ is given by 
\begin{equation}\label{eq: Laplace}
    \Delta = \frac{4}{2e^{u}}\frac{\partial^{2}}{\partial z \partial\bar{z}} = \frac{1}{2e^{u}} \left(\frac{\partial^{2}}{\partial x^{2}} + \frac{\partial^{2}}{\partial y^{2}} \right). 
\end{equation}
Here $\Delta=-\Delta_H^{(0)}$, where $\Delta_H^{(0)}$ denotes the nonnegative Hodge Laplacian on functions.
The Gaussian curvature can be expressed by the following formula:
\begin{equation*}
    K = -\frac{4}{2e^{u}} \frac{\partial^{2} \operatorname{ln} \sqrt{2e^{u}} }{\partial z \partial\bar{z}} = -e^{-u} u_{z \bar{z}}.
\end{equation*}
Moreover, since the induced metric of the $\R$-equivariant minimal Lagrangian surfaces in $Q_{2}$ and $Q_{2}^{*}$ depends only on $x$, we have $u_{z \bar{z}} = (1/4) u_{xx}$. Hence, the absolute total curvature is given by the following formula:
\begin{equation}\label{eq: total curvature}
    \int_{M} |K| dA  =  \frac{1}{2}\int_{M} |u_{xx}| dx\,dy,
\end{equation}
where $dA$ is the area element with respect to the induced metric. 

In this subsection, \(M\) denotes the maximal equivariant domain,
namely \(\mathbb C\) in the \(Q_2\) case and $\Sigma$ in the \(Q_2^*\) case, or, when the immersion closes in the \(y\)-direction, the corresponding quotient by a closing period.
\begin{Theorem}
Let $\{f^\lambda\}_{\lambda\in \mathbb{S}^1}$ be the associated family of
$\mathbb R$-equivariant minimal Lagrangian immersions into $Q_2$ or $Q_2^*$
constructed from the equivariant potential \eqref{eq: Delaunay potential},
with $a,b\in\mathbb R^*$ and $c\in\mathbb R$. Assume, as in Remarks~\ref{Rmk: flat case}
and~\ref{Rmk: timelike axes}, that
\[
\begin{cases}
c^2+(|a|-|b|)^2>0, & \text{in the $Q_2$ case},\\
c^2>(|a|+|b|)^2, & \text{in the $Q_2^*$ case}.
\end{cases}
\]
For each $\lambda\in\mathbb{S}^1$, let $M$ be the maximal
equivariant domain or, if $f^\lambda$ closes in the $y$-direction,
a quotient of that domain by a $y$-closing period.
Then the absolute total curvature of $f^\lambda$ on $M$ is
infinite:
\[
    \int_{M}|K|\,dA=+\infty .
\]
\end{Theorem}

\begin{proof}
Although the induced metrics in $Q_{2}$ and $Q_{2}^{*}$ have the
same form, the corresponding functions $v=v(x)$ satisfy different
ordinary differential equations. Therefore, we discuss the two cases
separately.

\textbf{For the case of $Q_{2}^{*}$}, let 
\[
    I=(-\kappa_{1}^{2},\kappa_{2}^{2})
\]
be the maximal interval on which $v$ is finite and nonzero.
For $\delta_{1},\delta_{2}>0$, set
\[
    I_{\delta_{1},\delta_{2}}
    =
    (-\kappa_{1}^{2}+\delta_{1},
      \kappa_{2}^{2}-\delta_{2}).
\]
Choose any interval $\Theta$ in the $y$-direction of finite positive
length $|\Theta|>0$. If the immersion closes in the $y$-direction,
$\Theta$ may be chosen to be a fundamental interval of the quotient;
otherwise, any bounded interval is sufficient. By \eqref{eq: total curvature},
\[
\begin{aligned}
    \int_{I_{\delta_{1},\delta_{2}}\times \Theta}
    |K|\,dA
    &=
    \frac{|\Theta|}{2}
    \int_{-\kappa_{1}^{2}+\delta_{1}}
         ^{\kappa_{2}^{2}-\delta_{2}}
    |u_{xx}|\,dx  \\
    &\ge
    \frac{|\Theta|}{2}
    \left|
        u_x(\kappa_{2}^{2}-\delta_{2})
        -
        u_x(-\kappa_{1}^{2}+\delta_{1})
    \right|.
\end{aligned}
\]
 We shall show that
\[
    u_x(x)\to-\infty
    \quad\text{at the left endpoint},\qquad
    u_x(x)\to+\infty
    \quad\text{at the right endpoint}.
\]
    Differentiating both sides of the conformal factor $2e^{u} = (1/2) \left(v(x)^{2}+16|ab|^{2}\,v(x)^{-2}\right)$ with respect to $x$, we have 
    \begin{equation}\label{eq: ux}
        u_{x} = \frac{2 v'}{v} \frac{v^{4} - 16|ab|^{2}}{v^{4} + 16|ab|^{2}}. 
    \end{equation}
Recall that \(v\) satisfies the equation \eqref{vx2} and we have 
\begin{equation*}
    c^{2} - a^{2} - b^{2} > 2|ab|. 
\end{equation*}
With $Y_{1}, Y_{2}$ defined above, the differential equation can be written as
\begin{equation*}
    (v')^{2} = (v^{2} + Y_{1})(v^{2} + Y_{2}) > 0.  
\end{equation*}
Therefore, 
\begin{equation*}
    v'(x) \ne 0 \qquad \text{for all}~x\in I. 
\end{equation*}
Hence $v$ is strictly monotone on the entire maximal interval $I$.
Moreover, by the real Jacobi elliptic representation \eqref{eq: solution v2}, the two
endpoints of the maximal interval \(I\) are finite: one corresponds to a zero and the other to an adjacent pole of the $\operatorname{sc}$-function. 
Thus, at a finite endpoint, either \(v(x)\to 0\)
or \(|v(x)|\to\infty\).

    \smallskip
\textrm{Case 1: \(v(x)\to 0\).} First of all, we have 
\begin{equation*}
    \frac{v^{4} - 16|ab|^{2}}{v^{4} + 16|ab|^{2}} \longrightarrow -1.
\end{equation*}
On the other hand, the equation \eqref{vx2} gives
\[
    (v')^2\to 16a^2b^2 >0.
\]
Therefore, we have $|v'| \rightarrow 4|ab| \ne 0$.  It follows that 
\begin{equation*}
    u_{x} \sim -\frac{2v'}{v}, 
\end{equation*}
and $|u_{x}| \rightarrow \infty$. 

\textrm{Case 2: \(|v(x)|\to\infty\).} First of all, we have 
\begin{equation*}
    \frac{v^{4} - 16|ab|^{2}}{v^{4} + 16|ab|^{2}} \longrightarrow 1.
\end{equation*}
On the other hand, expanding \eqref{vx2}, the dominant term as \(|v|\to\infty\) is
\(v^{4}\), so
\[
    (v')^{2}
    = v^{4} - 4(a^{2}+b^{2}-c^{2})\,v^{2} + 16a^{2}b^{2} = v^{4} + O(v^{2})
\]
for all \(|v|\) sufficiently large; hence \(|v'| \sim v^{2}\). It follows that 
\begin{equation*}
    u_{x} \sim \frac{2v'}{v}, 
\end{equation*}
and $|u_{x}|\sim 2|v| \rightarrow \infty$. 

Finally, since $v$ never vanishes on the connected interval $I$, it has a fixed sign. Replacing $v$ by $-v$ if necessary, we may assume $v>0$ on $I$. Thus, the sign of $u_{x}$ has the following relation with that of $v'$. 
\begin{equation*}
    \operatorname{sgn} (u_{x}) = -\operatorname{sgn}(v'),~\text{as}~v(x) \rightarrow 0; \quad \operatorname{sgn} (u_{x}) = \operatorname{sgn}(v'),~\text{as}~|v(x)| \rightarrow \infty.
\end{equation*}
Near the left endpoint, if $v \rightarrow 0$, then $v' > 0$; if $v \rightarrow \infty$, then $v' < 0$.
In both cases, we have $u_{x} < 0$. Similarly, we can show that $u_{x} > 0$ near the right endpoint. 
It follows that 
\[
    \left|
        u_x(\kappa_{2}^{2}-\delta_{2})
        -
        u_x(-\kappa_{1}^{2}+\delta_{1})
    \right|
    \longrightarrow+\infty
\]
as $\delta_{1},\delta_{2}\to0^{+}$. Consequently,
\[
    \int_M |K|\,dA=+\infty.
\]
Therefore, the absolute total curvature is infinite in the $Q_{2}^{*}$ case.

\textbf{For the case of $Q_{2}$}, the function $v(x)$ is of the form \eqref{eq: solution v}. 
By the Jacobi-\(\operatorname{dn}\) representation \eqref{eq: solution v},
and excluding the flat case, \(v\) is a bounded
nonconstant periodic function. Consequently, \(u\) is periodic.
  
 We claim that $u$ is nonconstant. Indeed, suppose that $u$ is constant. Then $v^{2} + 16|ab|^{2}\,v(x)^{-2}$ is constant. Let $s = v^{2} > 0$. Then it follows that 
\begin{equation*}
    s + \frac{16|ab|^{2}}{s} = C_{0}, 
\end{equation*}
for some constant $C_{0}$. Hence $s(x)$ takes values among the roots of the quadratic equation 
\begin{equation*}
    s^{2} - C_{0} s + 16|ab|^{2} = 0, 
\end{equation*}
which has at most two roots. Since $s(x)$ is continuous, it must take only one of these values and is therefore constant. Thus $v^{2}$ is constant. Since $v$ never vanishes, $v$ itself is constant, which is a contradiction. Therefore, $u$ is nonconstant. 

Suppose that $u_{xx} \equiv 0$. Then $u(x) = Ax + B$ for some constants $A$ and $B$. The periodicity of $u$ forces $A = 0$, and hence $u$ is constant, again a contradiction. Therefore, $u_{xx}$ cannot vanish identically. Let $T_{0}>0$ be a period of $u$. Then $u_{xx}$ is $T_{0}$-periodic. Since $u_{xx}$ cannot vanish identically, it follows that 
\[
    C_{T_{0}}:=
    \int_{0}^{T_{0}}|u_{xx}|\,dx>0.
\]
For every $N\in\mathbb N$, periodicity gives
\[
    \int_{0}^{N T_{0}}|u_{xx}|\,dx
    =
    \sum_{j=0}^{N-1}
    \int_{jT_{0}}^{(j+1)T_{0}}|u_{xx}|\,dx
    =
    N C_{T_{0}}.
\]
Thus
\[
    \int_{0}^{+\infty}|u_{xx}|\,dx=+\infty.
\]

Again, choose an interval $\Theta$ in the $y$-direction with $0<|\Theta|<\infty$. For a closing immersion, $\Theta$ may be chosen as a fundamental interval, while for a non-closing immersion any bounded interval may be used. Then
\[
    \int_M |K|\,dA
    \ge
    \int_{[0,+\infty)\times \Theta}|K|\,dA
    =
    \frac{|\Theta|}{2}
    \int_{0}^{+\infty}|u_{xx}|\,dx
    =+\infty.
\]
Therefore, the absolute total curvature is infinite also in the $Q_{2}$ case.
\end{proof}

\begin{Remark}
    In the excluded case \(c=0\) and \(|a|=|b|\), the conformal factor is constant and \(K\equiv0\). Hence its absolute total curvature is zero rather than infinite.
\end{Remark}

\subsection{Stability of the equivariant surfaces}
In this subsection, stability means stability in the usual
Riemannian sense with respect to compactly supported normal
variations.

\begin{Corollary}
    Let $M$ be a minimal Lagrangian surface in $Q_{2}^{*}$. Then $M$ is stable.     
\end{Corollary}
\begin{proof}
Let $V\in\Gamma_c(T^\perp M)$ be an arbitrary compactly supported normal variational vector field, and let
$\alpha_V=f^*(V\rfloor\omega)$ be the corresponding
one-form. Choose a compact smooth domain $U\subset M$ such that
\[
    \operatorname{supp}V\subset\operatorname{int}U.
\]
The second variation may be computed on $U$, using a variation
that is fixed near $\partial U$.

Since $\operatorname{Ric}_{Q_2^*}=-4g_{Q_{2}^{*}}$, the second variation
formula of Oh \cite[Theorem~3.5]{Oh1990} gives
\begin{align*}
    \mathscr A_V''(0)
    &=
    \int_U
    \left(
        \langle\Delta_H\alpha_V,\alpha_V\rangle
        -
        \operatorname{Ric}_{Q_2^*}(V,V)
    \right)dA\\
    &=
    \int_U
    \left(
        |d\alpha_V|^2
        +
        |\delta\alpha_V|^2
        +
        4|\alpha_V|^2
    \right)dA
    \geq 0,
\end{align*}
where $\Delta_H=d\delta+\delta d$ and $|\alpha_V|=|V|$.
The integration by parts produces no boundary term because
$\alpha_V$ vanishes near $\partial U$.
Since $V$ was an arbitrary compactly supported normal variation,
$M$ is stable.
\end{proof}

\begin{Corollary}
Let $f^\lambda:M\to Q_2$ be an $\mathbb R$-equivariant
minimal Lagrangian immersion constructed from the equivariant
potential \eqref{eq: Delaunay potential}, where $M$ is either
the maximal equivariant domain $\mathbb C$ or a $y$-closing
quotient $\mathbb C/(iL\mathbb Z)$, $L>0$.
Then $f^\lambda$ is unstable.
\end{Corollary}

\begin{proof}
On the maximal equivariant domain, $M$ is conformally equivalent
to $\mathbb C$. If the immersion closes in the $y$-direction
with period $L>0$, the map
\[
    [z]\longmapsto \exp\left(\frac{2\pi z}{L}\right)
\]
identifies $\mathbb C/(iL\mathbb Z)$ biholomorphically with
$\mathbb C^*$.
Thus, in either case, $M$ is an oriented Riemann surface of
finite topological type: it is conformally equivalent to
$\mathbb{S}^{2}\setminus\{\infty\}$ or $\mathbb{S}^{2}\setminus\{0,\infty\}$.
Moreover, $M$ is complete by
Proposition~\ref{prop: complete}.

Suppose, to the contrary, that $M$ is stable.
Since $Q_2$ is K\"ahler and $\operatorname{Ric}_{Q_2}=4g_{Q_{2}}$, the theorem of Miyaoka and Ueki
\cite[Theorem~2]{MU2014} applies.

If $M$ is conformally equivalent to $\mathbb C$, case~(2)
of that theorem requires that there be no constant $\kappa>0$
such that $\operatorname{Ric}_{Q_2}\geq\kappa g_{Q_{2}}$. This contradicts $\operatorname{Ric}_{Q_2}=4g_{Q_{2}}$.

If $M$ is conformally equivalent to
$\mathbb C^*=\mathbb{S}^{2}\setminus\{0,\infty\}$, its conformal type
corresponds to $s=2$ and $l=0$, which is excluded by the same
theorem.

Therefore neither conformal type can occur for a stable
minimal Lagrangian surface in $Q_2$, and the conclusion follows.
\end{proof}

\section{Hamiltonian stability and instability of compact cylindrical domains}
\label{sec: Hamiltonian stability}
In this section, we study Hamiltonian stability and instability of
compact cylindrical domains of the $y$-closing surfaces in $Q_2$.
Working directly with the Hamiltonian second variation and the sharp
Poincar\'e inequality on an interval, we obtain explicit sufficient
conditions on the width for stability and instability. We also show
that, for a fixed left endpoint and closing period, the stable widths
form an interval with a finite geometric critical width, and determine
this critical width explicitly in the flat case.

Let $f^{\lambda_0}:M_L=\mathbb C/(iL\mathbb Z)\rightarrow Q_2$ 
be a $y$-closing $\mathbb R$-equivariant minimal Lagrangian
immersion constructed from an equivariant potential with parameters
$a,b\in\mathbb R^*$ and $c\in\mathbb R$, where $L>0$ is a closing
period. The flat case is allowed in this section whenever the
immersion closes in the $y$-direction. Set
\[
    S:=a^2+b^2+c^2, \qquad \chi(x):=\frac12\left(v(x)^2+16|ab|^2v(x)^{-2}\right).
\]
Then equation \eqref{vx1} can be rewritten as
\begin{equation*}
    (v')^2=-v^4+4Sv^2-16|ab|^2.
\end{equation*}
Dividing this equation by $2v^2$ and comparing with the
definition of $\chi$, we obtain
\[
    \chi(x)
    =
    2S-\frac12\left(\frac{v'(x)}{v(x)}\right)^2.
\]
Together with the AM--GM inequality, this yields
\begin{equation}\label{eq:direct-chi-bounds}
    4|ab|\leq\chi(x)\leq 2S.
\end{equation}

Fix an arbitrary left endpoint $x_1\in\mathbb R$, and write
\[
    s:=x-x_1,
    \qquad
    \chi_1(s):=\chi(x_1+s).
\]
We consider
\[
    \Omega_{\ell,L}:=[0,\ell]\times(\mathbb R/L\mathbb Z),
    \qquad \ell>0,
\]
which corresponds to $[x_1,x_1+\ell]\times(\mathbb R/L\mathbb Z)$
in the original coordinates. The induced metric and area element on $\Omega_{\ell,L}$ are
\[
    ds^{2}_{M_{L}}=\chi_1(s)(ds^2+dy^2),
    \qquad
    dA=\chi_1(s)\,ds\,dy.
\]
Here $\ell$ is the width in the specified conformal coordinate,
rather than an intrinsic distance.

As in Section~\ref{Equivariant setup}, we consider Hamiltonian variations that are fixed
in a neighbourhood of $\partial\Omega_{\ell,L}$. Write $h=f_V\in C^\infty(\Omega_{\ell,L};\mathbb R)$ for the Hamiltonian potential. Then the admissibility condition is
\begin{equation}\label{eq:direct-admissible}
    \alpha_V=dh,
    \qquad
    \operatorname{supp}(dh)\Subset\Omega_{\ell,L}^{\circ}.
\end{equation}
In the coordinate expressions below, $h$ is understood to be
$L$-periodic in $y$. Thus $h$ is constant near each boundary circle, although the two
constants need not agree. Since the immersion is Lagrangian, the
bundle map
\[
    V\longmapsto \alpha_V=f^*(V \rfloor \omega)
\]
identifies the normal bundle with the cotangent bundle. Hence every
such exact one-form $dh$ determines a unique Hamiltonian normal field
$V$, with
\[
    \operatorname{supp}V=\operatorname{supp}(dh)
    \Subset\Omega_{\ell,L}^{\circ}.
\]

For an admissible $h$, let $\mathscr Q_\ell[h]:=\mathscr A_V''(0)$. 
Since the immersion, $x_1$, and $L$ are fixed, we omit them from the notation.
Hamiltonian stability means $\mathscr Q_\ell[h]\geq0$ for every
admissible $h$, while instability means that $\mathscr Q_\ell[h]<0$
for at least one such $h$. Put
\[
    \nabla_0h:=(h_s,h_y),
    \qquad
    \Delta_0h:=h_{ss}+h_{yy}.
\]
Since
\[
    \Delta_H h= -\Delta h =-\chi_1^{-1}\Delta_0h,
    \qquad
    |dh|^2=\chi_1^{-1}|\nabla_0h|^2,
\]
the second variation formula \eqref{eq: second variation formula} and
$\operatorname{Ric}_{Q_2}=4g_{Q_{2}}$ give
\begin{equation}\label{eq:direct-hamiltonian-variation}
    \mathscr Q_\ell[h]
    =\int_{\Omega_{\ell,L}}
      \left((\Delta_H h)^2-4|dh|^2\right)\,dA 
    =\int_0^\ell\int_0^L
      \left(
          \frac{(\Delta_0h)^2}{\chi_1(s)}
          -4|\nabla_0h|^2
      \right)\,dy\,ds.
\end{equation}

\medskip
\noindent\emph{The sharp interval Poincar\'e inequality.}
Consider the Dirichlet eigenvalue problem on the Euclidean interval $(0,\ell)$:
\[
    -q''=\rho q,
    \qquad
    q(0)=q(\ell)=0.
\]
Its eigenvalues are
\[
    \rho_n=\frac{n^2\pi^2}{\ell^2},
    \qquad n\geq1,
\]
with corresponding eigenfunctions $\sin(n\pi s/\ell)$.
Thus the first Dirichlet eigenvalue is $\pi^2/\ell^2$.
By its Rayleigh variational characterization \cite{C1984},
\begin{equation}\label{eq:interval-rayleigh}
    \inf_{\substack{p\in C_c^\infty(0,\ell)\\p\neq0}}
    \frac{\displaystyle\int_0^\ell|p'|^2\,ds}
         {\displaystyle\int_0^\ell|p|^2\,ds}
    =
    \frac{\pi^2}{\ell^2}.
\end{equation}
Consequently,
\begin{equation}\label{eq:sharp-interval-poincare}
    \int_0^\ell|p'|^2\,ds
    \geq
    \frac{\pi^2}{\ell^2}\int_0^\ell|p|^2\,ds,
    \qquad p\in C_c^\infty(0,\ell).
\end{equation}
The constant is optimal: any larger constant valid for every
$p\in C_c^\infty(0,\ell)$ would contradict
\eqref{eq:interval-rayleigh}.

\begin{Theorem}[Sufficient width conditions]
\label{thm:direct-hamiltonian-widths}
With the notation above, $\Omega_{\ell,L}$ is Hamiltonian stable
with respect to variations fixed near its boundary whenever
\begin{equation}\label{eq:direct-stable-width}
    0<\ell\leq\frac{\pi}{2\sqrt{2S}}.
\end{equation}
It is Hamiltonian unstable with respect to the same class of
variations whenever
\begin{equation}\label{eq:direct-unstable-width}
    \ell>\frac{\pi}{4\sqrt{|ab|}}.
\end{equation}
These sufficient conditions hold for every left endpoint $x_1$
and every closing period $L>0$.
\end{Theorem}

\begin{proof}
\emph{Stability.}
Let $h$ be an arbitrary admissible potential. Since $dh$ vanishes
near the boundary, both $h_s$ and $h_y$ vanish near $s=0,\ell$.
Integration by parts in $s$ and then in $y$ yields
\begin{align*}
    \int_0^\ell\int_0^L h_{ss}h_{yy}\,dy\,ds
    &=\left[\int_0^L h_sh_{yy}\,dy\right]_{s=0}^{s=\ell}
      -\int_0^\ell\int_0^L h_sh_{syy}\,dy\,ds\\
    &=-\int_0^\ell[h_sh_{sy}]_{y=0}^{y=L}\,ds
      +\int_0^\ell\int_0^L h_{sy}^2\,dy\,ds\\
    &=\int_0^\ell\int_0^L h_{sy}^2\,dy\,ds.
\end{align*}
Here the boundary term in $s$ vanishes because $h_s=0$ at both
ends, while the boundary term in $y$ cancels by periodicity.
Consequently,
\begin{equation}\label{eq:direct-flat-hessian-identity}
    \int_0^\ell\int_0^L(\Delta_0h)^2\,dy\,ds
    =\int_0^\ell\int_0^L
      \left(h_{ss}^2+2h_{sy}^2+h_{yy}^2\right)\,dy\,ds.
\end{equation}
For each fixed $y$, apply \eqref{eq:sharp-interval-poincare} to
$p(s)=h_s(s,y)$ and $p(s)=h_y(s,y)$, respectively. This gives
\[
    \int_0^\ell h_{ss}^2\,ds
    \geq\frac{\pi^2}{\ell^2}\int_0^\ell h_s^2\,ds,
    \qquad
    \int_0^\ell h_{sy}^2\,ds
    \geq\frac{\pi^2}{\ell^2}\int_0^\ell h_y^2\,ds.
\]
Integrating in $y$ and discarding nonnegative terms in
\eqref{eq:direct-flat-hessian-identity}, we obtain
\begin{equation}\label{eq:direct-derivative-poincare}
    \int_0^\ell\int_0^L(\Delta_0h)^2\,dy\,ds
    \geq\int_0^\ell\int_0^L(h_{ss}^2+h_{sy}^2)\,dy\,ds
    \geq\frac{\pi^2}{\ell^2}
      \int_0^\ell\int_0^L|\nabla_0h|^2\,dy\,ds.
\end{equation}
Using $\chi_1\leq2S$ in
\eqref{eq:direct-hamiltonian-variation}, followed by
\eqref{eq:direct-derivative-poincare}, we conclude that
\[
    \mathscr Q_\ell[h]
    \geq\left(\frac{\pi^2}{2S\ell^2}-4\right)
      \int_0^\ell\int_0^L|\nabla_0h|^2\,dy\,ds.
\]
This is nonnegative under \eqref{eq:direct-stable-width}.

\medskip
\emph{Instability.}
We prove the necessary condition
\[
    \Omega_{\ell,L}\text{ Hamiltonian stable}
    \quad\Longrightarrow\quad
    \ell\leq\frac{\pi}{4\sqrt{|ab|}}.
\]
Suppose that $\Omega_{\ell,L}$ is Hamiltonian stable. For an
arbitrary $\eta\in C_c^\infty(0,\ell)$, define
\[
    h(s,y):=-\int_s^\ell\eta(t)\,dt.
\]
Then $dh=\eta(s)\,ds$ has compact support in the interior, so $h$
is admissible. By stability and $\chi_1\geq4|ab|$,
\begin{equation*}
    0
    \leq\frac1L\mathscr Q_\ell[h]
      =\int_0^\ell\left(
        \frac{|\eta'|^2}{\chi_1(s)}-4|\eta|^2\right)\,ds
    \leq\frac1{4|ab|}\int_0^\ell|\eta'|^2\,ds
      -4\int_0^\ell|\eta|^2\,ds.
\end{equation*}
Thus
\[
    \int_0^\ell|\eta'|^2\,ds
    \geq16|ab|\int_0^\ell|\eta|^2\,ds
    \qquad
    \text{for every }\eta\in C_c^\infty(0,\ell).
\]
Taking the infimum of the corresponding Rayleigh quotients
and using \eqref{eq:interval-rayleigh}, we obtain
\[
    16|ab|
    \leq
    \inf_{\substack{\eta\in C_c^\infty(0,\ell)\\\eta\neq0}}
    \frac{\displaystyle\int_0^\ell|\eta'|^2\,ds}
         {\displaystyle\int_0^\ell|\eta|^2\,ds}
    =
    \frac{\pi^2}{\ell^2}.
\]
Hence Hamiltonian stability necessarily implies
\[
    \ell\leq\frac{\pi}{4\sqrt{|ab|}}.
\]
The contrapositive proves the instability assertion.
\end{proof}

The two explicit estimates do not determine the stability behavior in
the intermediate range
\begin{equation*}
    \frac{\pi}{2\sqrt{2S}}
    <\ell\leq
    \frac{\pi}{4\sqrt{|ab|}},
\end{equation*}
whenever this interval is nonempty. For fixed $x_1$ and
$L$, however, the domains satisfy
\[
    \Omega_{\ell_1,L}\subset\Omega_{\ell_2,L}
    \qquad\text{whenever }0<\ell_1<\ell_2.
\]
As shown in the following corollary, this nesting property implies that
the stable widths are separated from the unstable widths by a
well-defined geometric critical width.

\begin{Corollary}[Existence of a geometric critical width]
\label{cor:geometric-critical-width}
Fix the immersion, $x_1$ and $L$. There is a unique finite number
$\ell_H=\ell_H(x_1,L)>0$ such that
\begin{equation}\label{eq: Hamiltonian stable}
    \Omega_{\ell,L}\text{ is Hamiltonian stable}
    \quad\Longleftrightarrow\quad
    0<\ell\leq\ell_H.
\end{equation}
Moreover,
\begin{equation}\label{eq:geometric-critical-width-bounds}
    0<\frac{\pi}{2\sqrt{2S}}
    \leq\ell_H
    \leq\frac{\pi}{4\sqrt{|ab|}}<\infty.
\end{equation}
\end{Corollary}

\begin{proof}
Let $0<\ell_1<\ell_2$, and let $h$ be an arbitrary admissible
potential on $\Omega_{\ell_1,L}$. Since $dh$ vanishes near the
right boundary, $h$ equals a constant $C$ in a neighbourhood of
$\{\ell_1\}\times(\mathbb R/L\mathbb Z)$. Define
\[
    \widetilde h(s,y):=
    \begin{cases}
        h(s,y), & 0\leq s\leq\ell_1,\\
        C, & \ell_1<s\leq\ell_2.
    \end{cases}
\]
Since $h$ is already constant near $s=\ell_1$, this extension is
smooth. Its differential has the same support as $dh$, viewed
inside the larger domain, so $\widetilde h$ is admissible on
$\Omega_{\ell_2,L}$. Its derivatives vanish on the added region,
and therefore
\[
    \mathscr Q_{\ell_2}[\widetilde h]
    =
    \mathscr Q_{\ell_1}[h].
\]
If $\Omega_{\ell_2,L}$ is Hamiltonian stable, then
$\mathscr Q_{\ell_2}[\widetilde h]\geq0$, and hence
$\mathscr Q_{\ell_1}[h]\geq0$. Since $h$ was an arbitrary admissible
potential on $\Omega_{\ell_1,L}$, the smaller domain is also stable.
Thus stability of a larger domain implies stability of every
smaller domain. The same extension preserves a negative second
variation, so a negative direction on a smaller domain persists
on every larger domain.

Define
\[
    \mathcal W
    :=
    \bigl\{
        \ell>0:
        \Omega_{\ell,L}\text{ is Hamiltonian stable}
    \bigr\},
    \qquad
    \ell_H:=\sup\mathcal W.
\]
Theorem~\ref{thm:direct-hamiltonian-widths}
shows that $\mathcal W$ is nonempty and bounded, and gives
\eqref{eq:geometric-critical-width-bounds}.
If $0<\ell<\ell_H$, the definition of the supremum ensures that
there exists $w\in\mathcal W$ with $\ell<w\leq\ell_H$.
Indeed, otherwise $\ell$ would be an upper bound for $\mathcal W$,
contradicting $\ell<\sup\mathcal W$.
Since $\Omega_{w,L}$ is stable, the preceding monotonicity implies
that $\Omega_{\ell,L}$ is also stable.
On the other hand, if $\ell>\ell_H$, then $\ell\notin\mathcal W$,
because $\ell_H$ is an upper bound for $\mathcal W$. Hence
$\Omega_{\ell,L}$ is Hamiltonian unstable.

Finally, we prove stability at the critical width $\ell_H$.
Let $h$ be any admissible potential on $\Omega_{\ell_H,L}$.
Since $dh$ vanishes in a neighbourhood of the right boundary,
there exists $\delta\in(0,\ell_H)$ such that $h$ is constant on
\[
    [\ell_H-\delta,\ell_H]\times(\mathbb R/L\mathbb Z).
\]
Choose $\ell'$ with
\[
    \ell_H-\delta<\ell'<\ell_H.
\]
Then the restriction
$h|_{\Omega_{\ell',L}}$ is still admissible on
$\Omega_{\ell',L}$. Moreover, the removed region
$[\ell',\ell_H]\times(\mathbb R/L\mathbb Z)$ contributes nothing
to the second variation, since $\nabla_0h=\Delta_0h=0$ there.
Hence
\[
    \mathscr Q_{\ell_H}[h]
    =
    \mathscr Q_{\ell'}
    [h|_{\Omega_{\ell',L}}]
    \geq0,
\]
where the inequality follows from the stability of
$\Omega_{\ell',L}$, since $\ell'<\ell_H$.
As $h$ was arbitrary, $\Omega_{\ell_H,L}$ is Hamiltonian stable.
Together with the preceding conclusions, this gives
\[
    \mathcal W=(0,\ell_H],
\]
and hence \eqref{eq: Hamiltonian stable} follows. The uniqueness of $\ell_H$ is immediate from its definition as $\sup\mathcal W$.
\end{proof}

\begin{Remark}
In the flat case $c=0$ and $|a|=|b|$, the critical width can be
determined explicitly. Indeed, we have $S=2a^2$, so the two explicit
bounds in \eqref{eq:geometric-critical-width-bounds} coincide.
Thus $\chi\equiv 4a^2$, and
\[
    \ell_H=\frac{\pi}{4|a|},
    \qquad
    \Omega_{\ell,L}\text{ is Hamiltonian stable}
    \quad\Longleftrightarrow\quad
    0<\ell\leq\frac{\pi}{4|a|}.
\]
In general, Corollary~\ref{cor:geometric-critical-width}
establishes the existence of the geometric critical width but does
not locate it more precisely within these two bounds. Although these bounds are independent of \(x_1\) and \(L\), the critical width \(\ell_H\) itself may still depend on \(x_1\) and \(L\).
\end{Remark}

{\bf AI Disclosure.}
During the preparation and revision of this work,
the author used ChatGPT for assistance with English-language editing and the 
development and refinement of parts of the mathematical
arguments. This included exploring proof strategies and examining
assumptions, boundary conditions, and intermediate
calculations.
The author critically reviewed and revised the AI-assisted
material and takes full responsibility for the mathematical
arguments, computations, and conclusions presented in the manuscript.
 
The author declares that there is no conflict of interest.

\bibliographystyle{plain}
\bibliography{mybib}

\end{document}